\documentclass{amsart}
\usepackage{color}

\newcommand{\F}{\mathcal F}
\newcommand{\U}{\mathcal U}
\newcommand{\mK}{\mathbb K}
\newcommand{\IN}{\mathbb N}
\newcommand{\IR}{\mathbb R}
\newcommand{\e}{\varepsilon}
\renewcommand{\phi}{\varphi}
\newcommand{\Lip}{\mathrm{Lip}}
\newcommand{\lev}{\mathrm{lev}}
\newcommand{\dowa}{{\downarrow}}
\newcommand{\upa}{{\uparrow}}
\newcommand{\suc}{\mathrm{succ}}
\newcommand{\w}{\omega}
\newcommand{\diam}{\mathrm{diam}}
\newcommand{\dom}{\mathrm{dom}}
\newcommand{\Ra}{\Rightarrow}
\newcommand{\bw}{\overline{\omega}}
\newcommand{\Fix}{\mathrm{Fix}}

\newtheorem{theorem}{Theorem}
\newtheorem{proposition}{Proposition}
\newtheorem{lemma}{Lemma}

\newtheorem{problem}{Problem}
\theoremstyle{definition}
\newtheorem{definition}{Definition}

\title[Detecting topological and Banach fractals]{Detecting topological and Banach fractals\\ among zero-dimensional spaces}

\author{Taras Banakh, Magdalena Nowak, and Filip Strobin}
\address{T.Banakh: Ivan Franko National University of Lviv (Ukraine) and  Jan Kochanowski University in Kielce (Poland)}
\email{t.o.banakh@gmail.com}

\address{M.Nowak: Institute of Mathematics, Jan Kochanowski University in Kielce, ul. \'Swietokrzyska 15, 25-406 Kielce, Poland}
\email{magdalena.nowak805@gmail.com}
\address{F.Strobin: Institute of Mathematics, \L\'od\'z University of Technology, W\'olcza\'nska 215, 93-005 \L\'od\'z, Poland, and Jan Kochanowski University in Kielce (Poland)}
\email{filip.strobin@p.lodz.pl}

\thanks{The work of the first, second and third authors was partially supported by  National Science Centre grants
DEC-2012/07/D/ST1/02087, DEC-2012/07/N/ST1/03551, and FUGA No. 2013/08/S/ST1/00541, respectively.}

\date{}

\begin{document}
\begin{abstract} A topological space $X$ is called a {\em topological fractal} if $X=\bigcup_{f\in\F}f(X)$ for a finite system $\F$ of continuous self-maps of $X$, which is {\em topologically contracting} in the sense that for every open cover $\U$ of $X$ there is a number $n\in\IN$ such that for any functions $f_1,\dots,f_n\in \F${,} the set $f_1\circ\dots\circ f_n(X)$ is contained in some set $U\in\U$. If, in addition, all functions $f\in\F$ have Lipschitz constant $<1$ with respect to some metric generating the topology of $X$, then the space $X$ is called a {\em Banach fractal}. It is known that each topological fractal is compact and metrizable.  We prove that a zero-dimensional compact metrizable space $X$ is  a topological fractal if and only if $X$ is  a Banach fractal if and only if $X$ is either uncountable or $X$ is countable and its scattered height $\hbar(X)$ is a successor ordinal. For countable compact spaces this classification was recently proved by M.~Nowak.
\end{abstract}
\maketitle

\section{Introduction}
One of the milestones of Theory of Fractals is the result of Hutchinson \cite{Hut} (cf. also Barnsley \cite{B} and Hata \cite{Ha}) saying that for any finite family $\F$ of Banach contractions of a complete metric space $X${,} there is a unique non-empty compact set $A_\F\subset X$ such that
$$
A_\F=\bigcup_{f\in\F}f(A_\F),
$$
and, moreover, for every nonempty and compact set $K$, the sequence of iterations $\F^{n}(K)$ (of the map $\F:\mK(X)\to \mK(X)$, $\F:D\mapsto \bigcup_{f\in\F}f(D)$, defined on the hyperspace $\mK(X)$ of nonempty compact subsets of $X$) is convergent to $A_\F$ with respect to the Hausdorff metric on $\mK(X)$. Finite families of mappings in this context are called \emph{iterated function systems} (briefly, IFSs), and compact sets $A_\F$ generated by IFSs are called \emph{IFS-attractors} or {\em deterministic fractals} in the sense of Hutchinson and Barnsley.

The Hutchinson-Barnsley theory attracted many mathematicians and has been deeply studied since early 80's. In the last years the following problem was discussed:

\begin{problem}
Detect compact metric spaces which are (or are not) attractors of IFSs consisting of Banach contractions (or weaker types of contractions), or which are (or are not) homeomorphic to such attractors.
\end{problem}

Such a problem was considered for example in \cite{BN}, \cite{BI}, \cite{CR}, \cite{D}, \cite{K}, \cite{KN}, \cite{Kw}, \cite{N}, \cite{NS}, \cite{S}.
In particular, M. Nowak \cite{N} proved that a compact countable space $X$ is homeomorphic to an IFS-attractor if and only if the scattered height $\hbar(X)$ of $X$ is a successor ordinal. The scattered height $\hbar(X)$ is defined for any scattered topological space $X$ as follows. Let us recall that a topological space $X$ is {\em scattered} if each subspace $A\subset X$ has an isolated point. The Baire Theorem guarantees that each countable complete metric space is scattered.

Let $X$ be a topological space. For a subset $A\subset X$ denote by $A^{(1)}$ the set of non-isolated points of $A$. Put $X^{(0)}=X$ and for every ordinal $\alpha$ define the $\alpha$-th derived set $X^{(\alpha)}$ by the recursive formula $$X^{(\alpha)}=\bigcap_{\beta<\alpha}\big(X^{(\beta)}\big)^{(1)}.$$
The intersection $X^{(\infty)}=\bigcap_{\alpha}X^{(\alpha)}$ of all derived sets has no isolated points and is called the {\em perfect kernel} of $X$. For each point $x\in X\setminus X^{(\infty)}$ there is a unique ordinal $\hbar(x)$ such that $x\in X^{(\hbar(x))}\setminus X^{(\hbar(x)+1)}$, called the {\em scattered height} of $x$. For a scattered topological space $X$ the perfect kernel $X^{(\infty)}$ is empty and the ordinal $\hbar(X)=\sup\{\hbar(x):x\in X\}$ is called the {\em scattered height} of $X$.
For a non-scattered space $X$ we put $\hbar(X)=\infty$ and assume that $\infty$ is larger than any ordinal number.

In this paper we shall extend Nowak's characterization of countable IFS-attractors to the class of all zero-dimensional compact metrizable spaces and shall detect zero-dimensional Hausdorff spaces which are homeomorphic to topological or Banach fractals.

Following \cite{BN}, we define a Hausdorff topological space $X$ to be a {\em topological fractal} if $X=\bigcup_{f\in\F}f(X)$ for a finite system of continuous self-maps of $X$, which is {\em topologically contracting} in the sense that for every open cover $\U$ of $X${,} there is $n\in\IN$ such that for any maps $f_1,\dots,f_n\in\F${,} the set $f_1\circ\dots\circ f_n(X)$ is contained in some set $U\in\U$. This definition implies that each topological fractal is compact and metrizable (since the family $\{f_1\circ \dots\circ f_n(X):n\in\IN,\;f_1,\dots,f_n\in\F\}$ is a countable network of the topology of $X$).

A compact metrizable space $X$ is called a {\em Banach fractal\/} if $X=\bigcup_{f\in\F}f(X)$ for a finite system $\F$ of self-maps which have Lipschitz constant $<1$ with respect to some metric that generates the topology of $X$. It follows that a space $X$ is a Banach fractal if and only if $X$ is homeomorphic to the attractor of an IFS consisting of Banach contractions on a complete metric space. It is easy to see that each Banach fractal is a topological fractal.  On the other hand, there exist examples of topological fractals which are not Banach fractals (see, e.g., \cite{BN} and \cite{NS}).
However, such examples do not exist among zero-dimensional spaces. We recall that a topological space $X$ is called {\em zero-dimensional} if it has a base of the topology consisting of closed-and-open subsets. In the following theorem we shall show that each zero-dimensional topological fractal is a Banach fractal and moreover, is a Banach ultrafractal.

A compact topological space $X$ is {called} a {\em Banach ultrafractal} if $X=\bigcup_{f\in\F}f(X)$ for a finite family $\F$ of continuous self-maps of $X$ such that the family $(f(X))_{f\in\F}$ is disjoint and for any $\e>0$ there is an ultrametric $d$ generating the topology of $X$ such that $\Lip(\F)=\sup_{f\in\F}\sup_{x\ne y}\frac{d(f(x),f(y))}{d(x,y)}<\e$. We recall that a metric $d$ on $X$ is called an {\em ultrametric} if it satisfies the strong triangle inequality $d(x,z)\le\max\{d(x,y),d(y,z)\}$ for any points $x,y,z\in X$. It is well-known \cite[\S I.7]{Ke} that the topology of a compact metrizable space $X$ is generated by an ultrametric if and only if $X$ is zero-dimensional.

So, for any compact metrizable space we have the implications
$$\centerline{Banach ultrafractal $\Ra$ Banach fractal $\Ra$ topological fractal,}$$ none of which can be reversed in general. However, for zero-dimensional compact metrizable spaces all these notions are equivalent. 

\begin{theorem}\label{main} For a zero-dimensional compact metrizable space $X${,} the following conditions are equivalent:
\begin{enumerate}
\item $X$ is a topological fractal;
\item $X$ is a Banach fractal;
\item $X$ is a Banach ultrafractal;
\item the scattered height $\hbar(X)$ of $X$ is not a countable limit ordinal\newline (so, $\hbar(X)$ is either $\infty$ or a countable successor ordinal).
\end{enumerate}
\end{theorem}

This theorem will be proved in Section~\ref{s4} after some preparatory work made in Sections~\ref{s2} and \ref{s3}.

\section{Normed height trees}\label{s2}

Theorem~\ref{main} will be proved with help of a representation of a given zero-dimensional compact metrizable space as the boundary of a normed height tree. So, in this section we shall recall the necessary information about (height) trees and their boundaries.

By a {\em tree} we shall understand a partially ordered set $T$ such that $T$ has the (unique) smallest element $\min T$ and for every $x\in T$ the {\em{lower set}} ${\dowa}x=\{y\in T:y\le x\}$ is finite and linearly ordered. For any two elements $x,y\in T$ the intersection ${\dowa}x\cap{\dowa}y$, being finite and linearly ordered, has the largest element, which will be denoted by $x\wedge y$.
For any element $x\in T$ of a tree $T$ its {\em{upper set}} ${\upa}x=\{y\in T:x\le y\}$ is a subtree of $T$.
 The set $\suc(x)$ of minimal elements of the set $\upa x\setminus\{x\}$ is called the set of {\em successors} of $x$.
 For a subset $A\subset T$ of a tree $T$ let $$\dowa A=\bigcup_{a\in A}\dowa a\mbox{ \ and \ }\upa A=\bigcup_{a\in A}\upa a$$be the {\em lower} and {\em upper sets of $A$}, respectively.

 By a {\em branch} of $T$ we understand any maximal linearly ordered subset of $T$. The set $\partial T$ of all branches of $T$ is called the {\em boundary} of the tree $T$. For two distinct branches $x,y\in\partial T$ let $x\wedge y=\max(x\cap y)\in T$.  For an element $x\in T$ put ${\Uparrow}x=\{L\in\partial T:x\in L\}$.

Let $\bw_1=\{-1,\infty\}\cup\w_1$ be the linearly ordered set of countable ordinals enlarged by two points $-1$ and $\infty$ such that $-1<\alpha<\infty$ for any $\alpha\in\w_1$.

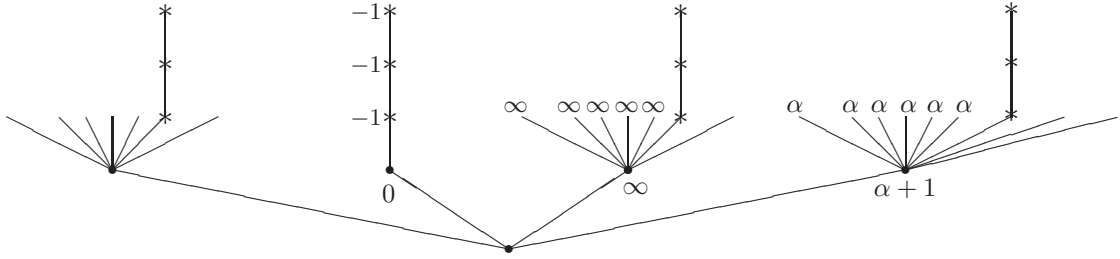
\begin{picture}(400,125)(-10,0)
\put(200,10){\circle*{3}}
\put(200,10){\line(5,1){150}}
\put(200,10){\line(3,2){45}}
\put(200,10){\line(-3,2){45}}
\put(200,10){\line(-5,1){150}}

\put(350,40){\circle*{3}}
\put(338,30){$\alpha+1$}
\put(350,40){\line(4,1){80}}
\put(350,40){\line(3,1){60}}
\put(350,40){\line(2,1){40}}
\put(350,40){\line(1,1){20}}
\put(350,40){\line(1,2){10}}
\put(350,40){\line(0,1){20}}
\put(350,40){\line(-1,2){10}}
\put(350,40){\line(-1,1){20}}
\put(350,40){\line(-2,1){40}}

\put(390,60){\line(0,1){40}}
\put(387,58){\Large $\ast$}

\put(387,78){\Large $\ast$}
\put(387,98){\Large $\ast$}
\put(369,62){$\alpha$}
\put(358,62){$\alpha$}
\put(348,62){$\alpha$}
\put(337,62){$\alpha$}
\put(326,62){$\alpha$}
\put(305,62){$\alpha$}

\put(50,40){\circle*{3}}
\put(50,40){\line(2,1){40}}
\put(50,40){\line(1,1){20}}
\put(50,40){\line(1,2){10}}
\put(50,40){\line(0,1){20}}
\put(50,40){\line(-1,2){10}}
\put(50,40){\line(-1,1){20}}
\put(50,40){\line(-2,1){40}}

\put(70,60){\line(0,1){40}}
\put(67,57){\Large $\ast$}
\put(67,77){\Large $\ast$}
\put(67,97){\Large $\ast$}

\put(245,40){\circle*{3}}

\put(243,31){$\infty$}
\put(245,40){\line(2,1){40}}
\put(245,40){\line(1,1){20}}
\put(245,40){\line(1,2){10}}
\put(245,40){\line(0,1){20}}
\put(245,40){\line(-1,2){10}}
\put(245,40){\line(-1,1){20}}
\put(245,40){\line(-2,1){40}}

\put(265,60){\line(0,1){40}}
\put(262,57){\Large $\ast$}
\put(262,77){\Large $\ast$}
\put(262,97){\Large $\ast$}

\put(250,62){\small $\infty$}
\put(240,62){\small $\infty$}
\put(229,62){\small $\infty$}
\put(218,62){\small $\infty$}
\put(198,62){\small $\infty$}

\put(245,40){\circle*{3}}

\put(155,40){\circle*{3}}
\put(155,40){\line(0,1){60}}
\put(152,57){\Large $\ast$}
\put(152,77){\Large $\ast$}
\put(152,97){\Large $\ast$}
\put(152,28){$0$}
\put(140,57){\small $-1$}
\put(140,77){\small $-1$}
\put(140,97){\small $-1$}
\end{picture}

\begin{definition}\label{d:ht} By {a} {\em height tree} we understand a pair $(T,\hbar)$ consisting of a tree $T$ and a function $\hbar:T\to \bw_1$ (called the {\em height function} on $T$) such that for every vertex $x\in T${,} the following conditions are satisfied:
\begin{itemize}
\item the set $\suc(x)$ contains exactly one point $*_x$ of height $\hbar(*_x)=-1$;
\item if $\hbar(x)\in\{-1,0\}$, then $\suc(x)=\{*_x\}$ and if $\hbar(x)>0$, then the set $\suc(x)$ is countably infinite;
\item if $\hbar(x)=\infty$, then all but finitely {many} points of the set $\suc(x)$ have height $\infty$;
\item if $0<\hbar(x)<\w_1$, then $\hbar(x)=\sup_{y\in\suc(x)}(\hbar(y)+1)=\lim_{y\in \suc(x)}(\hbar(y)+1)$ (which means that for any ordinal $\alpha<\hbar(x)$ there exists a finite subset $F\subset\suc(x)$ such that $\alpha<\hbar(y)+1\le \hbar(x)$ for all $y\in\suc(x)\setminus F$).
\end{itemize}
\end{definition}

For every vertex $x\in T$ of a height tree $T${,} the unique point $*_x\in\suc(x)$ with $\hbar(*_x)=-1$ will be called the {\em central point} of $x$. For each point $x\in T\setminus\{\min T\}$ we shall also define the {\em neighbor point} $\star_x$ letting $\star_x=*_y$ where $y\in T$ is the unique point such that $x\in\suc(y)$. For $x=\min T$ the neighbor point has not been defined, so we put $\star_x:=*_x$.

\begin{definition}\label{canontop} The boundary $\partial T$ of a height tree $T$ carries the {\em canonical topology} generated by the subbase $$\{{\Uparrow}x:x\in T,\;\hbar(x)\ne-1\}\cup\{\partial T\setminus{\Uparrow} x:x\in T\},$$ where ${\Uparrow}x:=\{\bar x\in \partial T:x\in\bar x\}$ for $x\in T$.
\end{definition}
\begin{definition}\label{d:tn}
Let $T$ be a height tree $T$. A {\em norm} $\|\cdot\|:T\to\IR$ on $T$ is a function having the following properties:
\begin{itemize}
\item for any vertices $x\le y$ of $T$ we get $\|x\|\ge \|y\|\ge 0$;
\item a vertex $x\in T$ has norm $\|x\|=0$ if and only if $\hbar(x)=-1$;
\item $\lim_{x\in T}\|x\|=0$, which means that for any positive real number $\e${,} the set $\{x\in T:\|x\|\ge \e\}$ is finite.
\end{itemize}
A {\em normed height tree} is a height tree{ $T$ endowed}  with a norm $\|\cdot\|$.
\end{definition}

The norm $\|\cdot\|$ of a normed height tree $T$ determines an ultrametric $d$ on $\partial T$ defined by
$$
d(x,y)=\begin{cases}\;
\max\{\|z\|:z\in (x\cup y)\cap\suc(x\wedge y)\}  & \mbox{if $x\neq y$}\\
\;0 & \mbox{if $x=y$}
\end{cases}
$$
for any branches $x,y\in\partial T$. This ultrametric will be called the {\em canonical ultrametric} on $\partial T$. It is easy to check that the canonical ultrametric $d$ generates the canonical topology on $\partial T$.
\smallskip


A compact metrizable space $X$ is called {\em unital} if $X$ is either uncountable or $X$ is countable and the derived set $X^{(\hbar(X))}$ is a singleton. It is easy to see that each compact metrizable space $X$ can be written as a finite topological sum of its unital subspaces. Indeed, if $X$ is uncountable, then $X$ is unital, and if $X$ is countable, then there are only finitely many point of  maximal scattered height. This, together with a known fact that each countable compact space is zero-dimensional gives us an appropriate division.

\begin{proposition}\label{p:Xtree} Each unital zero-dimensional compact metrizable space $X$ is homeomorphic to the boundary $\partial T$ of some normed height tree $T$ such that $\hbar(\min T)=\hbar(X)$.
\end{proposition}

\begin{proof} Fix an ultrametric $d$ of diameter $\le 1$ generating the topology of the compact zero-dimensional space $X$. Let $\U_0=\{X\}$.
We shall inductively construct a sequence $(\U_n)_{n\in\w}$ of disjoint closed covers of $X$ and a sequence of maps $(*_n:\U_n\to X)_{n\in\w}$ such that for every $n\in\w$ and $U\in\U_n${,} the following conditions are satisfied:
\begin{enumerate}
\item $*_n(U)\in U^{(\hbar(U))}$;
\item $\U_{n+1}(U)=\{V\in\U_{n+1}:V\subset U\}$ is a disjoint cover of $U$ by unital compact spaces;
\item $\{*_n(U)\}\in\U_{n+1}(U)$;
\item any set $V\in\U_{n+1}(U)$ distinct from $\{*_n(U)\}$ is closed-and-open in $X$;
\item each neighborhood of $*_n(U)$ contains all but finitely many sets of the cover $\U_{n+1}(U)$;
\item each set $V\in\U_{n+1}(U)$ has $\diam(V)\le 2^{-n}$.
\end{enumerate}

Assume that for some $n\in\w$ a cover $\U_n$ and a map $*_n:\U_n\to X$ have been constructed.
For every set $U\in\U_n$ we shall construct a closed cover $\U_{n+1}(U)$ satisfying the conditions (2)--(6) of the inductive construction. By our assumption, $U$ is a unital space. If $U$ is countable, then put $*_n(U)$ be the unique point of the singleton $U^{(\hbar(U))}$.
If $U$ is uncountable, then let $*_n(U)$ be any point of the perfect kernel $U^{(\infty)}$ of $U$.

Take any disjoint finite cover $\F(U)$ of the compact zero-dimensional space $U\subset X$ by closed-and-open subsets of diameter $\le 2^{-n}$. Find a unique set $V_0\in\F(U)$ containing the point $*_n(U)$ and choose a neighborhood base $\{V_n\}_{n=1}^\infty$ at $*_n(U)$ consisting of closed-and-open sets such that $V_n\subset V_{n-1}$ for every $n\in\IN$. If $\hbar(V{_0})<\w_1$, then we can replace $(V_n)_{n=1}^\infty$ by a suitable subsequence and assume that $\lim_{n\to\infty}(\hbar (V_n\setminus V_{n+1})+1)=\hbar(V_0)=\hbar(U)$ (here we assume that $\infty+1=\infty$).
Consider the cover $$\U_{n+1}'(U)=\big(\F(U)\setminus \{V_0\}\big)\cup\{V_n\setminus V_{n+1}:n\in\w\}\cup \big\{\{*_n(U)\}\big\}$$of $U$. For every $V\in\U_{n+1}'(U)$ find a finite disjoint cover $\U_{n+1}(V)$ of $V$ by unital compact spaces such that
$\hbar(V')=\hbar(V)$ for all $V'\in\U_{n+1}(V)$,
and put  $\U_{n+1}(U)=\bigcup_{V\in\U'_{n+1}(U)}\mathcal \U_{n+1}(V)$. It is clear that the cover $\U_{n+1}(U)$ satisfies the conditions (2)--(6) of the inductive construction.

Let $\U_{n+1}=\bigcup_{U\in\U_n}\U_{n+1}(U)$. For every (unital) space $V\in\U_{n+1}$ choose a point $*_{n+1}(V)\in V^{(\hbar(V))}$. This completes the inductive construction.
\smallskip


Now we construct a normed height tree $T$ whose boundary $\partial T$ is homeomorphic to $X$.
Let $$T=\{(n,U):n\in\w,\;U\in\U_{n}\}.$$ Given two pairs $(n,U),(m,V)\in T$ we write $(n,U)\le (m,V)$ if $n\le m$ and $V\subset U$. Define a height function $\hbar:T\to\bw_1$ by the formula
$$\hbar(n,U)=\begin{cases}
-1,&\mbox{ if $n>0$ and $U=\{\star_{n-1}(U)\}$},\\
\hbar(U),&\mbox{ otherwise}.
\end{cases}
$$
Here $\hbar(U)$ stands for the scattered height of the space $U$ (which is equal to $\infty$ if $U$ is uncountable) and $\star_{n-1}(U)=*_{n-1}(\tilde U)$ where $\tilde U\in\U_{n-1}$ is a unique set in $\U_{n-1}$ containing $U$.

Define a norm $\|\cdot\|$ on $T$ letting $$\|(n,U)\|=\diam (U\cup\{\star_{n-1}(U)\})\mbox{ for $(n,U)\in T$}.$$ Here we assume that $\star_{-1}(U)=*_0(U)$ for any set $U\in\U_0=\{X\}$.

 The norm $\|\cdot\|$ determines a canonical ultrametric generating the canonical topology on the boundary $\partial T$ of $T$. {We claim that the map
$$h:X\to\partial T,\;\;h:x\mapsto\{(n,U):n\in\w,\;x\in U\},$$
is a homeomorphism.

Indeed, observe that each branch $y\in\partial T$ consists of the decreasing nested sequence of compact sets with {vanishing} diameters, which implies that $\bigcap_{(n,U)\in y}U=\{h^{-1}(y)\}$ is a singleton. So, $h$ is a bijection. To see that $h$ is continuous, it suffices to check that for each subbasic open set $B\subset \partial T$ the preimage $h^{-1}(B)$ is open in $X$. By Definition~\ref{canontop} of the canonical topology on $\partial T$, the subbasic set $B$ is of the form $B=\partial T\setminus{\Uparrow}t$ for $t=(n,U)\in T$ or $B={\Uparrow}t$ for $t=(n,U)\in T$ with $\hbar(t)\ne-1$. In the first case the set $h^{-1}(B)=X\setminus U$ is open since the set $U\in\U_n$ is closed; in the second case the set $h^{-1}(B)=U$ is open since the set $U\in\U_n$ has $\hbar(n,U)\ne -1$ and hence is open in $X$ by the condition (4) of the inductive construction.} Being a continuous bijective map between compact Hausdorff spaces, the map $h:X\to\partial T$ is a homeomorphism.
\end{proof}

\section{Morphisms between height trees}\label{s3}

In this section we shall discuss morphisms between height trees.

\begin{definition}\label{d:hm}
For height trees $T,S$ a map $f:T\to S$ is called a {\em height morphism}
if for every $x\in T$ the following conditions are satisfied:
\begin{itemize}
\item $\hbar(f(x))\le\hbar(x)$,
\item $f(\suc(x))\subset\suc(f(x))$ and $f(*_x)=*_{f(x)}$,
\item for each $y\in\suc(f(x))\setminus \{*_{f(x)}\}$ there is at most one element $z\in \suc(x)\setminus \{*_x\}$ such that $y=f(z)$.
\end{itemize}
\end{definition}

Each height morphism $f:T\to S$ of height trees induces a map $\bar f:\partial T\to\partial S$ of their boundaries. The map $\bar f$ assigns to each branch $b\in\partial T$ of $T$ the unique branch of $S$ containing the linearly ordered set $f(b)=\{f(x):x\in b\}$. We are interested in finding conditions under which $\bar f$ is Banach contracting with respect to the canonical ultrametrics on $\partial T$, $\partial S$ generated by suitable norms on the height trees $T$ and $S$.

\begin{definition} Let $T,S$ be normed height trees. A height morphism $f:T\to S$ is called {\em $\lambda$-Lipschitz} for a real constant $\lambda$ if $\|f(x)\|\le\lambda\cdot\|x\|$ for each $x\in T$.
\end{definition}

The definition of the height morphism and the canonical ultrametric on the boundary of a normed height tree implies:

\begin{lemma}\label{l:Lip} Let $T,S$ be normed height trees and $\lambda$ be a positive real constant. For each $\lambda$-Lipschitz height morphism $f:T\to S$, the induced boundary map $\bar  f:\partial T\to\partial S$ is $\lambda$-Lipschitz with respect to the canonical ultrametrics on $\partial T$ and $\partial S$.
\end{lemma}

The following lemma will help us to construct surjective height morphisms between height trees.

\begin{lemma}\label{l:sur} For any height trees $T,S$ with $\hbar(\min T)\ge \hbar(\min S)$ there exists a surjective height morphism $f:T\to S$.
\end{lemma}

\begin{proof} This lemma will be proved by transfinite induction on the ordinal $\hbar(\min T)\in\bw_1$. It is trivial if $\hbar(\min T)\in\{-1,0\}$. Assume that for some ordinal $\alpha<\w_1$ the lemma has been proved for any height trees $T,S$ with  $\hbar(\min S)\le\hbar(\min T)<\alpha$. Take any height trees $T,S$ with  $\hbar(\min S)\le\hbar(\min T)=\alpha$.
If $\hbar(\min S)\in\{-1,0\}$, then the tree $S$ is order isomorphic to $\w$ and by a trivial reason, there exists a unique surjective height morphism $f:T\to S$. So, we assume that $\hbar(\min S)>0$, which implies that the set $\suc(\min S)$ is infinite.

By Definition~\ref{d:ht}, the roots $t=\min T$ and $s=\min S$ of the trees $T$ and $S$ have heights $$\hbar(t)=\sup\{\hbar(y)+1:y\in\suc(t)\}\ge \hbar(s)=
\sup\{\hbar(x)+1:x\in\suc(s)\},$$
which allows us to construct a bijective map $f_0:\dom(f_0)\to\suc(s)\setminus \{*_s\}$ defined
on a countable infinite subset $\dom(f_0)$ of $\suc(t)\setminus \{*_t\}$ such that $\hbar(f_0(x))\le\hbar (x)$ for any $x\in \dom(f_0)$. Extend the map $f_0$ to a map $\bar f_0:\suc(t)\to\suc(s)$ letting $\bar f_0\big(\suc(t)\setminus\dom(f_0)\big)=\{*_s\}$. Observe that for every $x\in\suc(t)$ we get $\hbar (\bar f_0(x))\le\hbar(x)<\alpha$. Then by the inductive assumption there exists a surjective morphism $f_x:{\upa} x\to{\upa}f_0(x)$ of the height subtrees of the trees $S$ and $T$. The morphisms $f_x$, $x\in\suc(t)$, compose a surjective morphism $f:T\to S$ defined by $f(t)=s$ and $f(z)=f_x(z)$ for any $x\in\suc(t)$ and $z\in{\upa}x$.

Finally, we consider the case of height trees $T,S$ with $\hbar(\min S)\le\hbar(\min T)=\infty$. For every $n\in\IN$ let $\lev^{-1}_T(n)=\{t\in T:|{\dowa}t|=n\}$ and $\lev^{-1}_S(n)=\{s\in S:|{\dowa}s|=n\}$ be the $n$-th levels of the trees $T$ and $S$, respectively. It follows that $\lev^{-1}_T(1)=\{\min T\}$ and $\lev^{-1}_S(1)=\{\min S\}$. Let $f_1:\lev^{-1}_T(1)\to \lev^{-1}_S(1)$ be a unique map. By induction for every $n\ge 2$ we can construct a map $f_n:\lev^{-1}_T(n)\to \lev^{-1}_S(n)$ such that for every $t\in \lev^{-1}_T(n-1)$ the following conditions are satisfied:
\begin{itemize}
\item $f_n(\suc(t))=\suc(f_{n-1}(t))$;
\item $f_n(*_t)=*_{f_{n-1}(t)}$;
\item for every point $y\in\suc(f_{n-1}(t))$ distinct from $*_{f_{n-1}(t)}$ there is a unique point $x\in\suc(t)$ with $y=f_n(x)$ and $\hbar(x)=\infty$;
\item each point $x\in \suc(t)$ with $\hbar(x)<\w_1$ has image $f_n(x)=*_{f_{n-1}(t)}$.
\end{itemize}
Then the union $f=\bigcup_{n=1}^\infty f_n:T\to S$ is the required surjective height morphism of $T$ onto $S$.
\end{proof}

For a self-map $f:X\to X$ of a set $X$ define its $n$-th iteration by induction: $f^0$ be the identity map of $X$ and $f^{n+1}=f^n\circ f$ for $n\in\w$. For a point $x\in X$ by $f^{<\w}(x)=\{f^n(x)\}_{n\in\w}$ we denote its orbit under the map $f$.

\begin{lemma}\label{l:unital} Assume that $X$ is either an uncountable zero-dimensional compact metrizable space $X$ or a unital countable space whose scattered height $\hbar(X)$ is a successor ordinal. Then $X$ is a Banach ultrafractal.
\end{lemma}

\begin{proof} There is nothing to prove if $|X|\le 1$. So, we assume that $|X|>1$ and hence $X$ is infinite (since $X$ is unital). By Proposition~\ref{p:Xtree}, $X$ is homeomorphic to the boundary $\partial T$ of a normed height tree $T$ with $\hbar(\min T)=\hbar(X)$. For any ordinal $\alpha$ put $\alpha-1=\sup\{\beta:\beta+1\le\alpha\}$. Observe that $\alpha-1=\alpha$ if $\alpha$ is limit. Also we put $\infty-1=\infty=\infty+1$.

 In the set $\suc(\min T)$ of successors of $\min T$, consider the countable subset
$$L=\{x\in\suc(\min T):\hbar(x)=\hbar(\min T)-1\}.$$ Definition~\ref{d:ht} guarantees that the set $F=\suc(\min T)\setminus \big(L\cup\{*_{\min T}\}\big)$ is finite and each point $x\in F$ has non-negative height $\hbar(x)<\hbar(\min T)-1$.

Write $L=\{x_n\}_{n\in\w}$ where $x_n\ne x_m$ for any distinct numbers $n,m\in\w$, and $x_\w=*_{\min T}$. Define a function $f_1:\suc(\min T)\to\suc(\min T)$  by the formula
$$f_1(x)=\begin{cases}
x_{n+1}&\mbox{ if $x=x_n$ for some $n\in\w$}\\
*_{\min T}&\mbox{otherwise}.
\end{cases}
$$
By Lemma~\ref{l:sur}, for every $x\in\suc(\min T)$ there is a surjective height morphism $f_x:{\upa x}\to \upa f_1(x)$. The height morphisms $f_x$, $x\in\suc(\min T)$, form a height morphism $f:T\to T$ such that $f(\min T)=\min T$ and $f|{\upa x}=f_x$ for every $x\in \suc(\min T)$.
Observe that $f(T)=\{\min T\}\cup\upa\{x_n\}_{1\le n\le\w}$ and the boundary map $\bar f:\partial T\to\partial T$ has image $\bar f(\partial T)=\{b\in\partial T:b\cap\{x_n\}_{1\le n\le\w}\ne\emptyset\}$.
\smallskip

Note that for every $x\in F\cup\{x_0\}$ we have $\hbar(\min T)\geq\hbar(x)$ so, by Lemma~\ref{l:sur}, there exists a surjective height morphism $g_x\colon T\to\upa x$. Then for every branch $b\in\partial T$ the set $\bar g_x(b):=\{\min T\}\cup g_x(b)$ is a branch of $\partial T$.

{Consider the finite family} $\F=\{f\}\cup\{g_x\colon x\in F\cup\{x_0\}\}$ of height morphisms on the tree $T$.
The definitions of the maps from $\F$ guarantee that $T$ is the disjoint union of $\{\phi(T)\}_{\phi\in\F}$, which implies that $\partial T$ is the disjoint union of $\{\bar\phi(\partial T)\}_{{\phi\in\F}}$.
To show that $X$ is a Banach ultrafractal, it remains for every $1>\lambda>0$ to find a norm $\|\cdot\|$ on $T$ which respect to which the morphisms from $\F$ are $\lambda$-Lipschitz.

Let $T_{-1}=\emptyset$, $T_0=\{\min T\}$ and $T_{n+1}=\F(T_n)=\bigcup_{\phi\in\F}\phi(T_n)$ for $n\in\w$. Note that
$$\bigcup_{n\in\w}T_n\supset\{x\in T\colon \hbar(x)\geq 0\}$$
which can be shown by the induction on the levels of the tree T.


The definitions of the morphisms from $\F$ guarantee that for every $n\in\w$ and a point $y\in T_{n+1}\setminus T_n$ of height $\hbar(y)\ne -1$ there is a unique point $x\in T_n$ such that $y\in\F({\{x\}})$. Then the function $\|\cdot\|:T\to[0,1]$ defined by $\|x\|=0$ if $\hbar(x)=-1$ and $\|x\|=\lambda^n$ if $x\in T_n\setminus T_{n-1}$ is a well-defined norm on $T$ turning the maps $\phi\in\F$ into $\lambda$-Lipschitz morphisms of $T$. By Lemma~\ref{l:Lip}, the induced boundary morphisms $\bar\phi:\partial T\to\partial T$ are $\lambda$-Lipschitz with respect to the canonical ultrametric on $T$ generated by the norm $\|\cdot\|$. Therefore $\partial T$ and $X$ are Banach ultrafractals.
\end{proof}

\section{Proof of Theorem~\ref{main}}\label{s4}

Given a zero-dimensional compact metrizable space $X$ we need to check the equivalence of the following conditions:
\begin{enumerate}
\item $X$ is a topological fractal;
\item $X$ is a Banach fractal;
\item $X$ is a Banach ultrafractal;
\item the scattered height $\hbar(X)$ of $X$ is not a countable limit ordinal.
\end{enumerate}

In fact, the implications $(3)\Ra(2)\Ra(1)$ are trivial. The (non-trivial) implication $(1)\Ra(4)$ follows from \cite{BKNNS}, \cite{Mihail} and \cite{N}. The first two papers show that a Hausdorff space is a topological fractal iff it is homeomorphic to an attractor of IFS consisting of weak (Edelstein) contraction. {This fact combined with a result from \cite{N} yields the implication  $(1)\Rightarrow (4)$}. The implication  $(4)\Ra(3)$ is proved in Lemma~\ref{l:unital} for unital spaces. Let us show that this implication holds also if $X$ is not unital. In this case $X$ is a compact countable space such that the set $X^{(\hbar(X))}$ contains more than one point. As was remarked earlier, $X$ can be written as a disjoint union $X=X_1\cup\dots\cup X_n$ of unital spaces with $\hbar(X_i)=\hbar(X)$ for all $i\le n$.  {Fix any positive $\lambda<1$.} By (the proof of) Lemma~\ref{l:unital}, each space $X_i$ can be written as a finite disjoint union $X_i=\bigcup_{f\in\F_i}f(X_i)$ for a family $\F_i$ of continuous selfmaps of $X_i$ with Lipschitz constant $\leq\lambda$, for a suitable ultrametric $d_i<1$ on $X_i$.
We define an ultrametric on $X$ as follows: for every $x,y\in X$,
$$d(x,y)=\begin{cases}
d_i(x,y),&\mbox{ $x,y\in X_i$},\\
1/\lambda,&\mbox{ otherwise}.
\end{cases}$$

The Banach Contracting Principle guarantees that for every $f_i\in\F_i$ the set $\Fix(f_i)=\{x\in X_i:f_i(x)=x\}$ is a singleton.
Extend the maps $f_i\in\F_i$ to continuous maps $\bar f_i:X\to X_i$ such that $\bar f_i(X\setminus X_i)=\Fix(f_i)$. Then the function system $\F=\bigcup_{1\le i\le n}\{\bar f_i\colon f_i\in\F_i\}$ witnesses that $X=\bigcup_{f\in \F}f(X)$ is a Banach ultrafractal.


\end{document}